\documentclass[11pt]{amsart}
\usepackage[colorlinks, citecolor=blue, dvipdfm, pagebackref]{hyperref}
\usepackage{extarrows}

\newtheorem{theorem}{Theorem}[section]
\newtheorem{lemma}[theorem]{Lemma}

\theoremstyle{definition}
\newtheorem{definition}[theorem]{Definition}

\newtheorem{example}[theorem]{Example}

\newtheorem{proposition}[theorem]{Proposition}
\newtheorem{corollary}[theorem]{Corollary}
\newtheorem{remark}[theorem]{Remark}
\newtheorem{conjecture}[theorem]{Conjecture}

\theoremstyle{remark}

\newcommand{\be}{\begin{equation}}
\newcommand{\ee}{\end{equation}}
\newcommand{\hooklongrightarrow}{\lhook\joinrel\longrightarrow}
\numberwithin{equation}{section}

\begin{document}

\title{Chern class inequalities on polarized manifolds and nef vector bundles}

\author{Ping Li}

\address{School of Mathematical Sciences, Tongji University, Shanghai 200092, China}

\email{pingli@tongji.edu.cn\\
pinglimath@gmail.com}
\author{Fangyang Zheng}
\address{School of Mathematical Sciences, Chongqing Normal University, Chongqing
401131, China}
\email{zheng@math.ohio-state.edu}
\thanks{The first author was partially supported by the National
Natural Science Foundation of China (Grant No. 11722109).}

%    General info

%\date{January 1, 2001 and, in revised form, June 22, 2001.}

%\dedicatory{}

\begin{abstract}
This article is concerned with Chern class and Chern number inequalities on polarized manifolds and nef vector bundles.
For a polarized pair $(M,L)$ with $L$ very ample, our first main result is a family of sharp Chern class inequalities. Among them the first one is a variant of a classical result %saying that $K_M+(n+1)L\geq0$ with equality if and only if $(M,L)=\big(\mathbb{P}^n,\mathcal{O}_{\mathbb{P}^n}(1)\big)$,
and the equality case of the second one is a characterization of hypersurfaces. The second main result is a Chern number inequality on it, which includes a reverse Miyaoka-Yau type inequality. The third main result is that the Chern numbers of a nef vector bundle over a compact K\"{a}hler manifold are bounded below by the Euler number. As an application, we classify compact K\"{a}hler manifolds with nonnegative bisectional curvature whose Chern numbers are all positive.
A conjecture related to the Euler number of compact K\"{a}hler manifolds with nonpositive bisectional curvature is proposed, which can be regarded as a complex analogue to the Hopf conjecture.
\end{abstract}

\maketitle

%\tableofcontents

\section{Introduction}\label{introduction}
Unless otherwise stated, vector bundles, complex manifolds and their dimensions mentioned throughout this article are respectively holomorphic, compact and positive.

Positivity is a central issue in complex differential geometry and algebraic geometry. For line bundles the positivity in differential geometry and ampleness in algebraic geometry are equivalent, thanks to the Kodaira embedding theorem. Griffiths (\cite{Gr69}) and Hartshorne (\cite{Hartshorne66}) respectively generalized these two notions to higher rank vector bundles of by introducing Griffiths-positivity and ampleness. It turns out that Griffiths-positivity implies ampleness, and Griffiths conjectured in \cite{Gr69} that these two notions are equivalent, which is true when the base manifold is a projective curve (\cite{CF}). In general constructing a Griffiths-positive Hermitian metric on an ample vector bundle seems to be quite difficult. Very recently Demailly proposed in \cite{De20} a method to attack this problem.

Griffiths also raised in \cite{Gr69} the question of characterizing the polynomials in the Chern classes and Chern forms for Griffiths-positive or ample vector bundles which are positive as cohomology classes and differential forms. At the cohomology class level this was completely answered by Fulton and Lazarsfeld (\cite{FL}), who showed that the set of such polynomials for ample vector bundles is exactly the cone generated by the Schur polynomials of Chern classes (an earlier special case was obtained by Bloch and Gieseker in \cite{BG}). This consequently implies that \emph{all the Chern numbers} of ample vector bundles are \emph{positive}. At the form level, Griffiths's question is still largely unknown except in some special cases (\cite{Gu1}, \cite{Gu2}). Recently the first author examined this question in \cite{Li20} for nonnegative Hermitian vector bundles in the sense of Bott and Chern (\cite{BC}) over (not necessarily K\"{a}hler) complex manifolds and showed that the Schur polynomials of Chern forms of such vector bundles are strongly nonnegative.

Griffiths-nonnegativity can also be defined and its counterpart in algebraic geometry is numerical effectiveness (``nefness" for short), where the former implies the latter. After some earlier works (\cite{Zheng89}, \cite{CP91}, \cite{CP93}), Demailly, Peternell and Schneider (\cite{DPS}) investigated in detail the structure of K\"{a}hler manifolds with nef tangent bundles. Among other things, they showed that those inequalities of Fulton-Lazarsfeld type remain true for Chern classes of nef vector bundles on K\"{a}hler manifolds (\cite[\S 2]{DPS}). As an application, they deduced that all the Chern numbers of a nef vector bundles on an $n$-dimensional K\"{a}hler manifolds are nonnegative and bounded from above by the Chern number $c_1^n$ (\cite[Coro. 2.6]{DPS}).

This upper bound plays a crucial role in establishing the main structural theorem in \cite{DPS} as well as in some other related applications. For instance, Zhang (\cite[Thm 3]{Zhang97}) applied it to show that the canonical line bundle of an immersed projective submanifold in an abelian variety is ample if and only if its signed arithmetic genus is positive. Yang (\cite[Thm 1.2]{Yang17}) applied it and some other results in \cite{DPS} to show that the holomorphic tangent bundle of a compact K\"{a}hler manifold with nonnegative holomorphic bisectional curvature is big if and only if it is Fano, and then classified such manifolds by using Mok's uniformization theorem (\cite{Mok88}).

Also motivated by this upper bound it was shown in \cite[Thm 3.2]{Li20} that for Bott-Chern nonnegative Hermitian vector bundles the Euler number and $c_1^r$ are respectively the lower and upper bounds at the \emph{form} level, by making use of the special properties of Bott-Chern nonnegativity.
Note that the condition of Bott-Chern nonnegativity is stronger than that of Griffiths nonnegativity and hence than that of nefness (\cite[Example 4.4]{Li20}). So the method in \cite{Li20} can \emph{not} be directly carried over to nef vector bundles, as remarked in \cite[Remark 3.3]{Li20}.

\emph{The main purposes} of this article are two-folded. The \emph{first main purpose} is to apply some nonnegativity results in \cite{Li20} by finding a good geometric model. To be more precise, given a polarized manifold $(M,L)$ with $L$ very ample, we associate it to a Bott-Chern nonnegative Hermitian vector bundle. This yields a family of sharp Chern class inequalities (Theorem \ref{Chern class inequality theorem}), among which the first one is a special case of a classical result. We apply some arguments of algebro-geometric nature to characterize the second equality case (Theorem \ref{level2equality}), which turns out to be a topological characterization of all hypersurfaces in complex projective spaces. A Chern number inequality involving $L$ and the first two Chern classes of $M$ is also deduced (Theorem \ref{Chern number inequality theorem}), which includes a reverse Miyaoka-Yau type inequality (Corollary \ref{reverse Miyaoka-Yau inequality}).

The \emph{second main purpose} is, by fully utilizing the positivity of Schur polynomials, we show that the Euler number of a nef vector bundle over a K\"{a}hler manifold is indeed the lower bound (Theorem \ref{Euler number inequality}). As a major application, we classify compact K\"{a}hler manifolds with nonnegative holomorphic bisectional curvature whose Chern numbers are \emph{all positive} (Theorem \ref{main application}). In view of Theorem \ref{main application}, a conjecture (Conjecture \ref{question}) related to the Euler number of compact K\"{a}hler manifolds with \emph{nonpositive} holomorphic bisectional curvature is proposed and we provide some positive evidences to it.

The rest of the article is organized as follows.  The aforementioned main results (Theorems \ref{Chern class inequality theorem}, \ref{level2equality}, \ref{Chern number inequality theorem}, \ref{Euler number inequality}) as well as some consequences are stated in Section \ref{main results}. In Section \ref{applications and examples} some applications including Theorem \ref{main application} and examples are presented.  We propose and discuss Conjecture \ref{question} in Section \ref{a conjecture}, which can be regarded as a ``dual version" to Theorem \ref{main application}. Sections \ref{first proof} and \ref{second proof} are then devoted to the proofs of Theorems \ref{Chern class inequality theorem} and \ref{Chern number inequality theorem} and Theorems \ref{Euler number inequality} and \ref{main application} respectively. Since the proof of Theorem \ref{level2equality} is a little more involved, we postpone it to the last section, Section \ref{third proof}.

\section{Main results}\label{main results}
Let $M^n$ be an $n$-dimensional projective manifold with $L$ an ample line bundle on it. A classical result (cf. \cite[p. 159]{BS}) states that $K_M+(n+1)L$ is always nef, and is ample unless $(M,L)=\big(\mathbb{P}^n,\mathcal{O}_{\mathbb{P}^n}(1)\big)$. Here $K_M$ is the canonical line bundle of $M$.
Our first main result is a family of sharp Chern class inequalities including a variant of this classical result as a special case. To state it, let us introduce some more notations first.

The \emph{Segre classes} of a vector bundle $E$ are defined to be formal inverse of the total Chern class of $E^{\ast}$, the dual of $E$, i.e.,
\be\label{segre class}s(E) = 1 + s_1(E) + s_2(E) + \cdots := c(E^{\ast})^{-1}=(1-c_1(E)+c_2(E)-\cdots)^{-1}.\ee
That is,
$$s_1(E) = c_1(E), \ \  s_2(E)=c_1^2(E)-c_2(E),\ \  s_3(E)=c_1^3(E)-2c_1(E)c_2(E)+c_3(E),\ \ldots,$$
and so on. For simplicity we denote by $s_i(M):=s_i(TM)$ and use $L$ for its first Chern class $c_1(L)$.  A real $(k,k)$-form $\varphi$ on $M^n$ ($1\leq k\leq n$) is called \emph{nonnegative} (\cite[Ch. 3, \S 1.A]{De1}) if it can be written as
$$\varphi=(\sqrt{-1})^{k^2}
\sum_i\psi_i\wedge\overline{\psi_i},$$
where these $\psi_i$ are $(k,0)$-forms. A cohomology class $\alpha\in H^{k,k}(M;\mathbb{R})$ is called \emph{nonnegative}, denoted by $\alpha\geq0$, if it contains a nonnegative $(k,k)$-form representative. Clearly if $\alpha\geq0$ then $\int_Y\alpha\geq0$ for any $k$-dimensional subvariety $Y\subset M$.

Let $L$ be a \emph{very ample} line bundle on $M$, $\text{dim}_{\mathbb{C}}H^0(M,L)=N+1$ and $\{s_0,s_1,\ldots,s_N\}$ a basis. Very ampleness means that $M$ can be holomorphically embedded into a complex projective space as a nonsingular projective variety via the following \emph{Kodaira map}:
\be\label{linear system}\begin{split}
M&\overset{i}{\hooklongrightarrow}
\mathbb{P}\big(H^0(M,L)^{\ast}\big)\cong\mathbb{P}^N\\
x&\longmapsto[s_0(x):s_1(x):\cdots:s_N(x)];\qquad i^{\ast}\big(\mathcal{O}_{\mathbb{P}^N}(1)\big)=L.
\end{split}
\ee
We remark that in this case the image $i(M)$ is \emph{nondegenerate} in the sense that it is not contained in a hyperplane. Otherwise some linear combination of $s_i$ would vanish on $M$, a contradiction.

With the notation above understood, it comes our first result.
\begin{theorem}\label{Chern class inequality theorem}
Let $L$ be very ample on $M^n$. We have for every $k\geq 1$ the following sharp Chern class inequalities
\be\label{Chern class inequality}
\sum_{i=0}^k(-1)^i\binom{n+k}{k-i}\cdot s_i(M)\cdot L^{k-i}\geq0,
\ee
where the equality case of (\ref{Chern class inequality}) occurs if
$(M,L)=\big(\mathbb{P}^n,\mathcal{O}_{\mathbb{P}^n}(1)\big)$ or $k>\min\{n,N-n\}$.
\end{theorem}
%\begin{remark}\label{dimensional reason}For $k>\min\{n,N-n\}$, the left hand of (\ref{Chern class inequality}) (trivially) vanishes due to dimensional reason, as we shall see in its proof in Section \ref{first proof}.\end{remark}

Note that for $k=1$, the left hand side of (\ref{Chern class inequality}) is $(n+1)L+K_M$. In this case,  the inequality $(n+1)L+K_M \geq 0$ is a special case of the aforementioned classical result, where $L$ is only needed to be ample, and the equality $K_M+(n+1)L=0$ occurs if and only if $(M,L)=\big(\mathbb{P}^n,\mathcal{O}_{\mathbb{P}^n}(1)\big)$ due to a classical result of Kobayashi and Ochiai (\cite{KO}).

For $k=2$, the left hand side of (\ref{Chern class inequality}) becomes
$$
\frac{1}{2}(n+2)(n+1)L^2-(n+2)Lc_1+c_1^2-c_2.$$
It turns out that the equality case of this gives exactly a characterization of hypersurfaces. In other words we have the following
\begin{theorem}\label{level2equality}
Let $L$ be very ample on $M^n$. Then by the $k=2$ case of Theorem \ref{Chern class inequality theorem} we have
\be\label{class inequality-1}
\frac{1}{2}(n+2)(n+1)L^2-(n+2)Lc_1+c_1^2-c_2 \geq 0.\ee
The equality occurs if and only if either $(M,L)=\big(\mathbb{P}^n,\mathcal{O}_{\mathbb{P}^n}(1)\big)$ or the Kodaira map (\ref{linear system}) embeds $M^n$ as a hypersurface in ${\mathbb P}^{n+1}$ (i.e., $N-n=1$) of degree $L^n$.
\end{theorem}

When $c_1=0$, i.e., for Calabi-Yau manifolds, the expressions in (\ref{Chern class inequality}) can be simplified a bit to lead to
\begin{corollary}\label{second corollary}
Let $M^n$ be a Calabi-Yau manifold with $n\geq 2$. Then for any very ample line bundle $L$ on $M^n$, it holds
\begin{eqnarray}
&& \label{c1=0}\frac{1}{2}(n+2)(n+1) L^2 - c_2 \ \geq \ 0; \\
&& \frac{1}{6}(n+3)(n+2)(n+1) L^3 - (n+3) Lc_2 - c_3 \ \geq \ 0.
\end{eqnarray}
The first equality case holds if and only if the Kodaira map (\ref{linear system}) embeds $M^n$ in ${\mathbb P}^{n+1}$, and in this case the degree of the hypersurface is necessarily $L^n=n+2$.
\end{corollary}

Our next result is the following Chern number inequality for $(M,L)$.
\begin{theorem}\label{Chern number inequality theorem}
Suppose $L$ is very ample on $M^n$. Then the following Chern number inequality holds:
\be\label{c1c2new}
\big[\frac{n(n+1)}{2}L^2-nc_1L+c_2\big]
\big[-c_1+(n+1)L\big]^{n-2}\leq\big[-c_1+(n+1)L\big]^{n}.
\ee
\end{theorem}

Fujita's very ampleness conjecture (\cite{Fuj}) asserts that $K_M+(n+2)L$ is very ample whenever $L$ is ample. This conjecture holds true if $L$ is further assumed to be globally generated (\cite[Thm 1.1]{Keeler}). So replacing $L$ in (\ref{c1c2new}) with $-c_1+(n+2)L$ and with some calculations (see Example \ref{calculate example} for details), we have
\begin{corollary}\label{reverse Miyaoka-Yau inequality}
Suppose $L$ is ample and globally generated on $M^n$. Then
\be\label{corollary inequality}
\big[-nc_1^2+2(n+1)c_2\big]
\big[-c_1+(n+1)L\big]^{n-2}\leq(n+2)^3\big[-c_1+(n+1)L\big]^{n}.
\ee
In particular, if $K_M$ is ample and globally generated, we have
\be\label{yau opposite}c_2(-c_1)^{n-2}\leq\frac{(n+2)^5+n}{2(n+1)}(-c_1)^n.\ee
\end{corollary}
\begin{remark}~
\begin{enumerate}
\item
Yau's celebrated Chern number inequality (\cite{Yau}) says that if $K_M$ is ample, then $$\frac{n}{2(n+1)}(-c_1)^n\leq c_2(-c_1)^{n-2},$$
to which (\ref{yau opposite}) can be viewed as a complement when $K_M$ is further assumed to be globally generated.
Note that the inequality (\ref{yau opposite}) is \emph{not} optimal due to the method we employ (see Corollary \ref{reverse Miyaoka-Yau inequality2} below). Even so, we are unaware of any this kind of \emph{reverse} Miyaoka-Yau type inequality in the literature, to the  best of our knowledge.

\item
In general, if $L$ is ample and $aL$ is very ample for some $a\in{\mathbb Q}^+$, we may from (\ref{c1c2new}) have an inequality involving an extra constant $a$. It is also known by the work of Demailly (\cite{Dem}) that $2K_M+12n^nL$ is always very ample for any ample $L$. So we may always take $a=2+12n^n$ when $K_M$ is ample, which of course also leads to a reverse Miyaoka-Yau inequality but with a very large constant.
\end{enumerate}
\end{remark}

As a direct corollary of Theorem \ref{level2equality}, we have the following
\begin{corollary}\label{reverse Miyaoka-Yau inequality2}
Let $M^n$ be projective manifold with $c_1<0$ (or $c_1>0$) and  $a\in\mathbb{Q}^+$ (or $a\in\mathbb{Q}^-$) such that $L=aK_M$ is a very ample line bundle. Replacing $L$ in (\ref{class inequality-1}) with $-ac_1$ (or $ac_1$) and multiplying by $(-c_1)^{n-2}$ (or $c_1^{n-2}$), it holds
\be\label{corollary inequality2}
\big[ \frac{1}{2}(n+2)(n+1)a^2 +(n+2)a +1\big] (\varepsilon c_1)^n \geq c_2 (\varepsilon c_1)^{n-2},
\ee
where $\varepsilon = -1$ (or $1$). If the equality holds, then (see Example \ref{example c1}) $\frac{1}{a}=\varepsilon(n+2-L^n)$, and the Kodaira map (\ref{linear system}) of $L=aK_M$ embeds $M^n$ as a hypersurface in ${\mathbb P}^{n+1}$ with degree $L^n$. In particular, for those with \emph{very ample} $K_M$, it holds that
\be\label{corollary inequality3}
 \frac{1}{2}(n^2+5n+8) (- c_1)^n \geq c_2 (- c_1)^{n-2},
\ee
and the equality occurs if and only if $M^n\subset {\mathbb P}^{n+1}$ is a hypersurface of degree $n+3$.
\end{corollary}

Next, for two given positive integers $k$ and $r$, let us denote by $\Gamma(k,r)$ the set of partitions $\lambda=(\lambda_1,\lambda_2,\ldots,\lambda_k)$ of weight $k$ by nonnegative integers $\lambda_j\leq r$:
\be\label{partition} r\geq\lambda_1\geq\lambda_2\geq\cdots\geq\lambda_k\geq0,\qquad
\sum_{j=1}^{k}\lambda_j=k.\ee

Let $E$ be a rank $r$ vector bundle. For each partition $\lambda\in\Gamma(k,r)$, the Chern monomial $c_{\lambda}(E)$ is defined by
$$c_{\lambda}(E):=\prod_{j=1}^kc_{\lambda_j}(E)\in H^{k,k}(M;\mathbb{Z}).$$

With this notation understood, we come to our third main result.
\begin{theorem}\label{Euler number inequality}
Let $E$ be a rank $r$ nef vector bundle on an $n$-dimensional compact K\"{a}hler manifold $(M,\omega)$, and $c_i(E)$ the $i$-th Chern classes of $E$ \rm($0\leq i\leq r$\rm). Then for each $1\leq k\leq n$ and $\lambda\in\Gamma(k,r)$ we have
\be\label{main inequality}\int_Mc_{\lambda}(E)\wedge[\omega]^{n-k}\geq\int_Mc_k(E)\wedge[\omega]^{n-k}
\geq0.\ee
In particular, all the Chern numbers of $E$ are bounded below by the (nonnegative) Euler number $\int_Mc_n(E)$.
\end{theorem}

\section{Applications and examples}\label{applications and examples}
\subsection{Some applications to Theorem \ref{Euler number inequality}}
Mori (\cite{Mori79}) and Siu-Yau's (\cite{SiuYau}) independent solution to the Frankel conjecture asserts that an $n$-dimensional compact K\"{a}hler manifold with positive holomorphic bisectional curvature is biholomorphic to $\mathbb{P}^n$. Building on a splitting result of Howard-Smyth-Wu (\cite{HSW}, \cite{Wu}) and combining analytic and algebraic tools, Mok solved the generalized Frankel conjecture by showing the following uniformization theorem.
If a compact K\"{a}hler manifold has nonnegative holomorphic bisectional curvature, then its universal cover is holomorphically isometric to
\be\label{uniformization of Mok}
(\mathbb{C}^q,g_0)\times(\mathbb{P}^{N_1},\theta_1)\times
\cdots\times(\mathbb{P}^{N_k},\theta_k)\times(M_1,g_1)\times\cdots
\times(M_p,g_p),\ee
where $g_0$ is flat, $\theta_i$ are K\"{a}hler metrics on $\mathbb{P}^{N_i}$ carrying nonnegative holomorphic bisectional curvature, and $(M_i,g_i)$ are irreducible compact Hermitian symmetric spaces of rank $\geq2$ equipped with the canonical metrics.

As a major application of Theorem \ref{Euler number inequality}, we can, with the help of Mok's uniformization theorem, \emph{classify} compact K\"{a}hler manifolds with nonnegative holomorphic bisectional curvature whose Chern numbers are \emph{all positive}.
\begin{theorem}\label{main application}
Let $M$ be an $n$-dimensional compact K\"{a}hler manifold with nonnegative holomorphic bisectional curvature. Then
\begin{enumerate}
\item
either all the Chern numbers of $M$ are positive, in which case $M$ is
holomorphically isometric to
\be\label{simply connected}(\mathbb{P}^{N_1},\theta_1)\times
\cdots\times(\mathbb{P}^{N_k},\theta_k)\times(M_1,g_1)\times\cdots
\times(M_p,g_p)\ee
with $\theta_i$ and $(M_i,g_i)$ are the same as those in (\ref{uniformization of Mok});
\item
or all the Chern numbers of $M$ vanish, in which case $\pi_1(M)$ is infinite and its universal cover splits off a nontrivial complex Euclidean factor $(\mathbb{C}^q,g_0)$ in (\ref{uniformization of Mok}).
\end{enumerate}
\end{theorem}
Combining the results in \cite{DPS} and \cite{Yang17} with Theorem \ref{main application}, we have now the following equivalent conditions to characterize simply-connected compact K\"{a}hler manifolds with nonnegative holomorphic bisectional curvature.
\begin{theorem}
Let $M$ be an $n$-dimensional compact K\"{a}hler manifold with nonnegative holomorphic bisectional curvature. Then the following four statements are equivalent.
\begin{enumerate}
\item
$M$ is Fano, i.e., $c_1(M)>0$;

\item
the holomorphic tangent bundle $T_M$ is big;

\item
the Chern number $c_1^n>0$;

\item
all the Chern numbers of $M$ are positive.
\end{enumerate}
\end{theorem}
\begin{remark}
``$(1)\Longleftrightarrow(3)$" is due to Demailly-Peternell-Schneider (\cite[\S 4]{DPS}), where they indeed showed this for compact K\"{a}hler manifolds with nef tangent bundles. ``$(2)\Longleftrightarrow(3)$" is due to Yang (\cite[Thm 1.2]{Yang17}). ``$(4)\Longleftrightarrow(3)$" follows from our Theorem \ref{main application}.
\end{remark}

Recall that, for a (possibly non-K\"{a}hler) Hermitian metric $g$ on a complex manifold $M$, the holomorphic bisectional curvature of $g$ can still be defined in terms of the Chern connection. Denote by $TM$ the holomorphic tangent bundle of $M$. Then the Hermitian vector bundle $(TM, g)$ \big(resp. ($T^{\ast}M$,g)\big) is Griffiths-nonnegative if and only if the holomorphic bisectional curvature of $g$  is nonnegative (resp. nonpositive) (\cite[p. 177]{Zheng00}). Since Griffiths-nonnegativity implies nefness, combining the lower bound in Theorem \ref{Euler number inequality} with the upper bound in \cite{DPS} we have the following
\begin{corollary}\label{coro1}
Let $M$ be an $n$-dimensional compact K\"{a}hler manifold with a (possibly different) Hermitian metric $g$ whose holomorphic bisectional curvature is nonnegative (resp. nonpositive). Then the Chern numbers $c_{\lambda}[M]$ of $M$ satisfy
$$0\leq c_n[M]\leq c_{\lambda}[M]\leq c_1^n[M]$$
$$\big(\text{resp.
$0\leq(-1)^nc_n[M]\leq(-1)^nc_{\lambda}[M]\leq(-1)^nc_1^n[M]$}\big).$$
\end{corollary}

The famous Hopf conjecture asserts that the Euler number $\chi(M)$ of a closed $2n$-dimensional Riemannian manifold $M$ with sectional curvature $K<0$ (resp. $K\leq0$) satisfies $(-1)^n\chi(M)>0$ \big(resp. $(-1)^n\chi(M)\geq0$\big), which is true when $n\leq2$ (\cite{Ch}) but is still open in its full generality for $n\geq 3$. Gromov (\cite{Gr}) introduced the notion of ``K\"{a}hler hyperbolicity", which includes K\"{a}hler metrics with $K<0$ as special cases, and showed that the Euler number of K\"{a}hler hyperbolic manifolds have the expected sign. This notion was extended independently by Cao-Xavier (\cite{CX}) and Jost-Zuo (\cite{JZ}) to the nonpositive case. These consequently settled the above Hopf Conjecture for K\"{a}hler manifolds. Indeed what they achieved is a solution of a stronger conjecture, the Singer conjecture in the K\"{a}hlerian case (\cite[\S 11]{Lu}).

Note that the sign of holomorphic bisectional curvature of a K\"{a}hler metric is dominated by that of (Riemannian) sectional curvature (\cite[p. 178]{Zheng00}). So our following corollary provides more information on Chern numbers of compact K\"{a}hler manifolds with nonpositive sectional curvature.
\begin{corollary}
Let $(M,\omega)$ be an $n$-dimensional compact K\"{a}hler manifold with nonpositive (Riemannian) sectional curvature. Then its Chern numbers $c_{\lambda}[M]$ satisfy
 $$0\leq(-1)^n\chi(M)=(-1)^nc_n[M]\leq(-1)^nc_{\lambda}[M]\leq(-1)^nc_1^n[M].$$
\end{corollary}

\begin{remark}
By refining Gromov's idea, the first author recently deduced that (\cite[Thm 2.1]{Li19}) $n$-dimensional K\"{a}hler hyperbolic manifolds indeed satisfy a family of optimal Chern number inequalities and the first one is exactly $(-1)^nc_n\geq n+1$ , which is an improved inequality expected by the Hopf conjecture.
\end{remark}

\subsection{Examples}
We give in this subsection some examples to illustrate some main results in Section \ref{main results}.

The following tensor product formulas for Segre classes and total Chern class are well-known (cf. \cite[p. 49-50, p. 56]{Ful2}) and shall be used in the sequel.
\begin{example}\label{tensor formula}
Let $E$ be a vector bundle of rank $n+1$ (resp. $n$) and $L$ a line bundle. Then we have
\be s_k(E\otimes L)=\sum_{i=0}^k\binom{n+k}{k-i}\cdot s_i(E)\cdot L^{k-i},\qquad \forall~k.\nonumber\ee
\be\Big(\text{resp. $c(E\otimes L)=\sum_{i=0}^nc_i(E)\cdot(1+L)^{n-i}$}.\Big)\nonumber\ee
Note that our Segre classes $s_k$ defined in (\ref{segre class}) are different from those in \cite{Ful2} by a sign $(-1)^k$ (cf. \cite[p. 50]{Ful2}).
\end{example}

\begin{example}\label{example1}
Let $L:=\mathcal{O}_{\mathbb{P}^n}(1)$ be the hyperplane line bundle on $\mathbb{P}^n$. We show that the equality cases of (\ref{Chern class inequality}) are satisfied by $(\mathbb{P}^n,L)$. Indeed we have the relation $T{\mathbb{P}^n}\oplus\underline{\mathbb{C}}=(n+1)L$, where $\underline{\mathbb{C}}$ denotes the trivial line bundle, and thus
\be\label{example}\underline{\mathbb{C}}^{n+1}
=(T^{\ast}\mathbb{P}^n\oplus\underline{\mathbb{C}})\otimes L.\ee
Taking the $k$-th Segre class $s_k(\cdot)$ on both sides of (\ref{example}) leads to
\be\label{example2}\begin{split}
0&=s_k\big((T^{\ast}\mathbb{P}^n\oplus\underline{\mathbb{C}})\otimes L\big)\\
&=\sum_{i=0}^{k}\binom{n+k}{k-i}\cdot s_i(T^{\ast}\mathbb{P}^n\oplus\underline{\mathbb{C}})\cdot L^{k-i}\qquad(\text{by Example \ref{tensor formula}})\\
&=\sum_{i=0}^{k}(-1)^i\binom{n+k}{k-i}\cdot s_i(\mathbb{P}^n)\cdot L^{k-i}.
\end{split}
\ee
%where the second equality in (\ref{example2}) is due to the tensor product formula for Segre classes (see Lemma \ref{tensor formula}).
\end{example}

\begin{example}\label{example c1}
\begin{enumerate}
\item
If $c_1=0$, (\ref{c1=0}) says that $\frac{1}{2}(n+2)(n+1)L^2\geq c_2$ for any very ample line bundle $L$, and equality occurs when and only when the Kodaira map (\ref{linear system}) for $L$ embeds $M^n$ as a degree $n+2$ hypersurface in ${\mathbb P}^{n+1}$.
\item
For a projective manifold $M^n$ with $c_1<0$, one may take a very ample line bundle $L=-ac_1$ with $a\in {\mathbb Q}^+$, and get from (\ref{class inequality-1}) that
$$\big[\frac{1}{2}(n+2)(n+1)a^2+(n+2)a+1\big]c_1^2\geq c_2,$$
with equality if and only if the Kodaira map for $L$ embeds $M^n$ as a hyperfurface in ${\mathbb P}^{n+1}$ of degree $L^n=a^n(-c_1)^n\geq n+3$. In this case we have $c_1=(n+2-L^n)L$ and so $a=1/[L^n-(n+2)].$

\item
Similarly, for a Fano manifold $M^n$, one may take a very ample line bundle $L=ac_1$ with $a\in {\mathbb Q}^+$ and get from (\ref{class inequality-1}) the inequality
$$\big[\frac{1}{2}(n+2)(n+1)a^2-(n+2)a+1\big]c_1^2\geq c_2,$$
where the equality case occurs if and only if either $L^n=1$ and $(M,L)=\big(\mathbb{P}^n,\mathcal{O}_{\mathbb{P}^n}(1)\big)$, or the Kodaira map embeds $M^n$ as a hyperfurface in ${\mathbb P}^{n+1}$ with degree $L^n\leq n+1$. In the latter case we have $a=1/[(n+2)-L^n].$
\end{enumerate}
\end{example}

\begin{example}\label{calculate example}
In this example we indicate how to derive (\ref{corollary inequality}) from (\ref{c1c2new}). Indeed, direct calculations imply that
\be\label{case}\begin{split}
&\frac{n(n+1)}{2}\big[-c_1+(n+2)L\big]^2-nc_1\big[-c_1+(n+2)L\big]+c_2\\
=\ \ &\frac{n(n+2)^2}{2(n+1)}\big[-c_1+(n+1)L\big]^2-\frac{n}{2(n+1)}c_1^2+c_2.
\end{split}\ee
Replacing $L$ in (\ref{c1c2new}) with $-c_1+(n+2)L$ and using (\ref{case}) it yields
\be\label{case2}
\begin{split}
&\Big\{\frac{n(n+2)^2}{2(n+1)}\big[-c_1+(n+1)L\big]^2-\frac{n}{2(n+1)}c_1^2+c_2\Big\}
\big[-c_1+(n+1)L\big]^{n-2}\\
\leq \ \ &(n+2)^2\big[-c_1+(n+1)L\big]^n.
\end{split}
\ee
Multiplying by $2(n+1)$ on both sides of (\ref{case2}) and cancelling the terms involving $[-c_1+(n+1)L]^n$ yields (\ref{corollary inequality}). Further replacing $L$ with $-c_1$ in (\ref{corollary inequality}) leads to (\ref{yau opposite}).
\end{example}

\section{A conjecture in the nonpositive case}\label{a conjecture}
In view of Theorem \ref{main application}, it is natural to wonder, for compact K\"{a}ler manifolds with nonpositive holomorphic bisectional curvature, whether their signed Chern numbers have the similar phenomenon of simultaneous positivity like Theorem \ref{main application}. In contrast to Mok's uniformization theorem in the nonnegative situation, our current knowledge on the structure of nonpositive holomorphic bisectional curvature compact K\"{a}hler manifolds is still much less satisfactory. So no appropriate structure theorem is available to deduce a similar conclusion, to our best knowledge. Nevertheless, we believe the validity of the following conjecture, which can be regarded as the complex analogue to the Hopf conjecture.
\begin{conjecture}\label{question}
Let $(M,\omega)$ be an $n$-dimensional compact K\"{a}hler manifold with nonpositive holomorphic bisectional curvature whose Ricci curvature is quasi-negative. Then the signed Euler number $(-1)^nc_n[M]>0$.
\end{conjecture}
\begin{remark}
If a K\"{a}hler metric $\omega$ has nonpositive bisectional curvature, then its Ricci curvature is nonpositive. So the quasi-negativity of Ricci curvature is equivalent to $(-1)^nc_1^n[M]>0$. In view of Corollary \ref{coro1}, Conjecture \ref{question} is \emph{equivalent to} the simultaneous positivity and vanishing of Chern numbers for such manifolds. Unfortunately, so far we are unable to solve it.
\end{remark}

A positive evidence to Conjecture \ref{question} indeed has been provided in \cite{Li20}. It turns out that the holomorphic cotangent bundles of (immersed) complex submanifolds in complex tori admit Bott-Chern nonnegative Hermitian metrics (\cite[Ex 4.3]{Li20}). As an application of the main results, it is shown in (\cite[Thm 7.3]{Li20}), among other things, that their signed Chern numbers have the phenomena of simultaneous positivity and vanishing. Note that complex submanifolds in complex tori can be equipped with K\"{a}hler metrics with nonpositive holomorphic bisectional curvature (the induced metrics from the flat complex tori). So this result indeed partially confirms Conjecture \ref{question}.

It is worth mentioning that the following splitting result for nonpositive bisectional curvature compact K\"{a}hler manifolds, which is dual to the famous splitting result of Howard-Smyth-Wu in the nonnegative situation (\cite{HSW}, \cite{Wu}) and originally conjectured by S.-T. Yau, has been recently confirmed by Liu (\cite{Liu14}) building on earlier works of Wu and the second author (\cite{Zheng02}, \cite{WuZheng}).
\begin{theorem}[Liu, Wu-Zheng]\label{LWZ}
If $(M,\omega)$ is an $n$-dimensional compact K\"{a}hler manifold with nonpositive holomorphic bisectional curvature whose maximal rank of the Ricci form is $r$ ($0\leq r\leq n$), then there exists a finite cover $M'$ of $M$ such that $M'$ is holomorphically isometric to a flat torus bundle $T^{n-r}$ over a compact K\"{a}hler manifold $N^{r}$ with nonpositive bisectional curvature and $c_1(N)<0$.
\end{theorem}
More precise statement and various corollaries can be found in \cite{Liu14}. Although it seems to us that the information provided by this splitting theorem is \emph{not} enough to reach Conjecture \ref{question} in its full generality, we can still have the following result.
\begin{proposition}
Let $(M,\omega)$ be as in Conjecture \ref{question}. Then the signed Chern number $(-1)^nc_2c_1^{n-2}[M]>0$. In particular, Conjecture \ref{question} is true for $n=2$.
\end{proposition}
\begin{proof}
Let $M'$ be the finite cover of $M$ as in Theorem \ref{LWZ}. The quasi-negativity of the Ricci curvature on $M$ implies that $c_1(M')<0$. Then Yau's Chern number inequality tells us that
\be\label{Yau Chern inequality}(-1)^nc_2c_1^{n-2}[M']\geq\frac{n}{2(n+1)}(-1)^nc_1^n[M'].\ee
Since Chern numbers are multiplicative under a finite cover, (\ref{Yau Chern inequality}) leads to
$$(-1)^nc_2c_1^{n-2}[M]\geq\frac{n}{2(n+1)}(-1)^nc_1^n[M],$$
from which the conclusion follows.
\end{proof}

\section{Proofs of Theorems \ref{Chern class inequality theorem} and  \ref{Chern number inequality theorem}}\label{first proof}
Let $(M^n,L)$ be a polarized manifold with $L$ a \emph{very ample} line bundle. This $L$ embeds $M$ into some complex projective space $\mathbb{P}^N$ as a nondegenerate smooth projective variety via the Kodaira map (\ref{linear system}).
%\be\label{Kodaira map}M\overset{i}{\hooklongrightarrow}
%\mathbb{P}\big(H^0(M,L)^{\ast}\big)\cong\mathbb{P}^N,\qquad N+1:=\text{dim}_{\mathbb{C}}H^0(M,L),\qquad i^{\ast}\big(\mathcal{O}_{\mathbb{P}^N}(1)\big)=L.\ee

The embedding $i$ induces a Gauss map $\gamma$ which sends $p\in M$ to $\mathbb{T}_pM$, the $n$-dimensional projective tangent space of $M$ at $p$ in $\mathbb{P}^N$:
\be\label{Gauss map}
\begin{split}
M&\overset{\gamma}{\longrightarrow}
\mathbb{G}_{n}(\mathbb{P}^N)\cong G_{n+1}(\mathbb{C}^{N+1})\\
p&\longmapsto\mathbb{T}_pM,
\end{split}\ee
where $\mathbb{G}_{n}(\mathbb{P}^N)$ is the Grassmannian variety of $n$-dimensional projective subspaces in $\mathbb{P}^N$, which is isomorphic to $G_{n+1}(\mathbb{C}^{N+1})$, the usual complex Grassmannian of $(n+1)$-dimemsional linear spaces in $\mathbb{C}^{N+1}$.

Let $S$ be the rank $n+1$ tautological subbundle over $G_{n+1}(\mathbb{C}^{N+1})$. The bundles $\gamma^{\ast}(S)$, $L$ and the tangent bundle $TM$ are famously related to each other via the following exact sequence (cf. \cite[p. 198]{At})
\be\label{exact seq}0\longrightarrow \underline{\mathbb{C}}\longrightarrow\gamma^{\ast}(S)\otimes L\longrightarrow TM\longrightarrow 0,\ee
where as before $\underline{\mathbb{C}}$ denotes the trivial line bundle on $M$.
\begin{remark}
The geometric model described above has been used in several papers to deduce Chern number inequalities in their context. For instance, Tai used in \cite{Tai} the exact sequence (\ref{exact seq}) and those symmetric polynomials invariant by translation to deduce Chern number inequalities for complete intersections in $\mathbb{P}^N$. This idea was further push forwarded by Manivel in \cite{Manivel}. Du and Sun applied it in \cite{DS} to treat the boundedness of the region given by the Chern ratios.
\end{remark}

\subsection{Proof of Theorem \ref{Chern class inequality theorem}}
Note that the rank $N-n$ universal quotient bundle $Q$ over $G_{n+1}(\mathbb{C}^{N+1})$ is globally generated and so is $\gamma^{\ast}(Q)$ over $M^n$, the pull back under the Gauss map $\gamma$ in (\ref{Gauss map}). Bott and Chern introduced in \cite{BC} a nonnegativity notion for holomorphic vector bundles, which is called \emph{Bott-Chern nonnegativity} in \cite{Li20}. Its precise definition is not important in this article, and we only need the fact that any globally generated vector bundle can be equipped with a Bott-Chern nonnegative Hermitian metric (\cite[(4.2)]{Li20}). So $\gamma^{\ast}(Q)$ admits such a Hermitian metric, say $h$. In this case the Chern forms $c_k\big(\gamma^{\ast}(Q),h\big)$ %$\big(1\leq k\leq\min\{n,N-n\}\big)$
with respect to the canonical Chern connection are nonnegative as real $(k,k)$-forms (\cite[Prop. 3.1]{Li20}). We remark that in \cite{Li20}
they are called ``strongly nonnegative" to distinguish from another weaker nonnegativity. So the Chern classes $c_k\big((\gamma^{\ast}(Q)\big)\geq0$ as cohomology classes.

We claim that
\be\label{key}c_k\big((\gamma^{\ast}(Q)\big)=\sum_{i=0}^k(-1)^i\binom{n+k}{k-i}\cdot s_i(M)\cdot L^{k-i},\ee
from which the inequalities (\ref{Chern class inequality}) follows. Since $\gamma^{\ast}(Q)$ is a rank $N-n$ vector bundle over an $n$-dimensional manifold. So $c_k\big((\gamma^{\ast}(Q)\big)=0$ for $k>\min\{n,N-n\}$.

We now prove (\ref{key}). Note that the total Chern classes of $Q$ and the universal subbundle $S$ are related by $c(Q)c(S)=1$ and therefore $c\big(\gamma^{\ast}(Q)\big)c\big(\gamma^{\ast}(S)\big)=1$. This implies that
\be\label{case3} c\big(\gamma^{\ast}(Q)\big)=\frac{1}{c\big(\gamma^{\ast}(S)\big)}=
s\big(\gamma^{\ast}(S^{\ast})\big)\ee
by the definition of Segre classes in (\ref{segre class}). On the other hand, the exact sequence (\ref{exact seq}) tells us that
\be\label{case4}s\big(\gamma^{\ast}(S^{\ast})\big)=
s\big[(\underline{\mathbb{C}}\oplus T^{\ast}M)\otimes L\big].\ee
Combining (\ref{case3}) with (\ref{case4}) yields
\be
\begin{split}
c_k\big((\gamma^{\ast}(Q)\big)&=s_k\big(\gamma^{\ast}(S^{\ast})\big)\\
&=s_k\big[(\underline{\mathbb{C}}\oplus T^{\ast}M)\otimes L\big]\\
&=\sum_{i=0}^k\binom{n+k}{k-i}
\cdot s_i(\underline{\mathbb{C}}\oplus T^{\ast}M)\cdot L^{k-i}\qquad(\text{by Example \ref{tensor formula}})\\
&=\sum_{i=0}^k(-1)^i\binom{n+k}{k-i}
\cdot s_i(M)\cdot L^{k-i}.
\end{split}
\ee
This completes the proof of (\ref{key}) and hence of the inequalities in (\ref{Chern class inequality}).

In Example \ref{example1} we have showed by directly calculation that the equality cases in (\ref{Chern class inequality}) are satisfied by the pair $(M,L)=\big(\mathbb{P}^n,\mathcal{O}_{\mathbb{P}^n}(1)\big)$.
This fact is now obvious from our proof as in this case the Gauss map $\gamma$ is a constant map and so $\gamma^{\ast}(Q)$ is trivial.

\subsection{Proof of Theorem \ref{Chern number inequality theorem}}
For simplicity we denote by $c_i$ the $i$-th Chern class of $M$. The exact sequence (\ref{exact seq}) says that the total Chern class of $\gamma^{\ast}(S^{\ast})$ is given by
\be\label{chern formula}
\begin{split}
c\big(\gamma^{\ast}(S^{\ast})\big)&=
(1+L)\cdot c(T^{\ast}M\otimes L)\\
&=(1+L)\big[\sum_{i=0}^n(-1)^ic_i\cdot(1+L)^{n-i}\big].
\qquad(\text{by Example \ref{tensor formula}})
\end{split}\nonumber
\ee
In particular
\begin{eqnarray}\label{c1c2}
\left\{\begin{array}{ll}
c_1\big(\gamma^{\ast}(S^{\ast})\big)=-c_1+(n+1)L\\
c_2\big(\gamma^{\ast}(S^{\ast})\big)=\frac{n(n+1)}{2}L^2-nc_1L+c_2.
\end{array} \right.
\end{eqnarray}

Note that $S^{\ast}$ is also globally generated and so is $\gamma^{\ast}(S^{\ast})$. As mentioned above $\gamma^{\ast}(S^{\ast})$ can be endowed with a Bott-Chern nonnegative Hermitian metric. Then we have the following Chern number inequality, which is a special case of \cite[Thm 3.2]{Li20}
\be\label{case5}c_2\big(\gamma^{\ast}(S^{\ast})\big)
\Big[c_1\big(\gamma^{\ast}(S^{\ast})\big)\Big]^{n-2}\leq
\Big[c_1\big(\gamma^{\ast}(S^{\ast})\big)\Big]^{n}.
\ee
Substituting (\ref{c1c2}) into (\ref{case5}) yields the desired inequality (\ref{c1c2new}). Note that (\ref{case5}) can also be deduced from \cite[Coro. 2.6]{DPS} as $\gamma^{\ast}(S^{\ast})$ is also nef.

\section{Proofs of Theorems \ref{Euler number inequality} and \ref{main application}}\label{second proof}
\subsection{Proof of Theorem \ref{Euler number inequality}}
Under the notations introduced in Section \ref{main results}, we start with the following
\begin{definition}
For each partition $\lambda=(\lambda_1,\lambda_2,\ldots,\lambda_k)\in\Gamma(k,r)$, the \emph{Schur polynomial} $S_{\lambda}(c_1,\ldots,c_r)\in\mathbb{Z}[c_1,\ldots,c_r]$ is defined as follows.
\be
\begin{split}
S_{\lambda}(c_1,\ldots,c_r):= \ &\text{det}(c_{\lambda_i-j+j})_{1\leq i,j\leq k}\qquad(\text{$i:$ row, $j:$ column})\\
= \ &
%\text{det}\left(\begin{array}{cccc}
\begin{vmatrix}
c_{\lambda_1} & c_{\lambda_1+1} &\cdots &c_{\lambda_1+k-1}\\
c_{\lambda_2-1} & c_{\lambda_2} &\cdots&c_{\lambda_2+k-2}\\
\vdots&\vdots&\ddots&\vdots\\
c_{\lambda_k-k+1}&c_{\lambda_k-k+2}&\cdots&c_{\lambda_k}
%\end{array}\right)
\end{vmatrix}
,\end{split}\nonumber\ee
where we adopt the convention that $c_0:=1$ and $c_i:=0$ if $i\notin[0,r]$.
\end{definition}
We shall use the following two special Schur polynomials.
\begin{example}
We have
\be\label{specialschur1}S_{(i,0,\ldots,0)}(c_1,\ldots,c_r)=c_i\ee and
\be\label{specialschur2}\begin{split}
S_{(k-i,i,0,\ldots,0)}(c_1,\ldots,c_r)&=
\begin{vmatrix}
c_{k-i} & c_{k-i+1} &\ast &\cdots &\ast\\
c_{i-1} & c_{i} &\ast&\cdots&\ast\\
0&0&1&\cdots&\ast\\
\vdots&\vdots&\vdots&\ddots&\\
0&0&0&\cdots&1
\end{vmatrix}\qquad \big(0\leq i\leq[\frac{k}{2}]\big)\\
&=c_{k-i}c_i-c_{k-i+1}c_{i-1}.
\end{split}\ee
\end{example}

Schur polynomials have appeared and played important roles in algebraic combinatorics, representation theory, algebraic geometry and so on. We refer to \cite{Ful} and \cite{Ma} for various facts on them. What we need in the proof is the following special case of the remarkable Littlewood-Richardson rule (\cite[p. 142]{Ma}).

\begin{lemma}\label{positivity of Schur}
Denote by
$$P(k,r):=\big\{\sum_{\lambda\in\Gamma(k,r)} a_{\lambda}S_{\lambda}(c_1,\ldots,c_r)~\big|~a_{\lambda}\geq0\,\big\}.$$
Then $$P(k_1,r)\cdot P(k_2,r)\subset P(k_1+k_2,r).$$
\end{lemma}

Now we are ready to prove Theorem \ref{Euler number inequality}.
For convenience, we denote by $$C_i:=c_i(E),\qquad S_{\lambda}:=S_{\lambda}\big(c_1(E),\ldots,c_r(E)\big).$$
The following Fulton-Lazarsfeld type inequalities for nef vector bundles over compact K\"{a}hler manifolds are due to Demailly, Peternell and Schneider (\cite[\S 2]{DPS}).
\be\label{DPS}\int_MS_{\lambda}\wedge[\omega]^{n-k}\geq0,
\qquad \Big(1\leq k\leq n,~\lambda\in\Gamma(k,r)\Big).\ee
In view of (\ref{DPS}) and the definition of $P(k,r)$, in order to prove Theorem \ref{Euler number inequality}, it suffices to show
\be\label{basic inequality}C_{\lambda_1}C_{\lambda_2}\cdots
C_{\lambda_k}-C_{\lambda_1+\cdots\lambda_k}\in P(\sum_{i=1}^k\lambda_i,r),\ee
which follows from the following two lemmas.

\begin{lemma}
We have
\be\label{claim1}C_{\lambda_1+\cdots\lambda_{t-1}}
C_{\lambda_t}-C_{\lambda_1+\cdots\lambda_t}\in P(\sum_{i=1}^t\lambda_i,r),\qquad\forall~2\leq t\leq k.\ee
\end{lemma}
\begin{proof}
\be
\begin{split}
C_{\lambda_1+\cdots\lambda_{t-1}}
C_{\lambda_t}-C_{\lambda_1+\cdots\lambda_t}
=&
\sum_{i=0}^{\lambda_t-1}\big(C_{\lambda_1+\cdots\lambda_{t-1}+i}
C_{\lambda_t-i}-C_{\lambda_1+\cdots\lambda_{t-1}+i+1}C_{\lambda_t-i-1}\big)\\
=&\sum_{i=0}^{\lambda_t-1}S_{(\lambda_1+\cdots\lambda_{t-1}+i,\lambda_t-i,0,\ldots,0)}
\qquad\big(\text{by (\ref{specialschur2})}\big)\\
\in& P(\sum_{i=1}^t\lambda_i,r).
\end{split}\nonumber
\ee
\end{proof}

\begin{lemma}
The inequality (\ref{basic inequality}) holds true.
\end{lemma}
\begin{proof}
\be\begin{split}
&C_{\lambda_1}C_{\lambda_2}\cdots
C_{\lambda_k}-C_{\lambda_1+\cdots+\lambda_k}\\
=&\sum_{i=1}^{k-1}\big(C_{\lambda_1+\cdots+\lambda_i}C_{\lambda_{i+1}}\cdots
C_{\lambda_k}-C_{\lambda_1+\cdots+\lambda_{i+1}}C_{\lambda_{i+2}}\cdots
C_{\lambda_k}\big)\\
=&\sum_{i=1}^{k-1}\big(C_{\lambda_1+\cdots+\lambda_i}C_{\lambda_{i+1}}-
C_{\lambda_1+\cdots+\lambda_{i+1}}\big)C_{\lambda_{i+2}}\cdots
C_{\lambda_k}\\
=&\sum_{i=1}^{k-1}\big(C_{\lambda_1+\cdots+\lambda_i}C_{\lambda_{i+1}}-
C_{\lambda_1+\cdots+\lambda_{i+1}}\big)S_{(\lambda_{i+2},0,\ldots,0)}\cdots
S_{(\lambda_k,0,\ldots,0)}.\qquad\big(\text{by (\ref{specialschur1})}\big)\\
\in&
\sum_{i=1}^{k-1}P(\lambda_1+\cdots+\lambda_{i+1},r)P(\lambda_{i+2},r)\cdots P(\lambda_k,r)
\qquad\big(\text{by (\ref{claim1})}\big)\\
\subset&P(\sum_{i=1}^k\lambda_i,r).\qquad\big(\text{by Lemma \ref{positivity of Schur}}\big)
\end{split}\nonumber
\ee
This completes the proof of (\ref{basic inequality}) and hence Theorem \ref{Euler number inequality}.
\end{proof}

\subsection{Proof of Theorem \ref{main application}}
Its proof is an application of the following Howard-Smyth-Wu's splitting theorem (\cite{HSW}, \cite{Wu}) and Mok's uniformization theorem (\cite{Mok88}).
\begin{theorem}
Let $(M,\omega)$ be an $n$-dimensional compact K\"{a}hler manifold with nonnegative holomorphic bisectional curvature. Let $q\in\mathbb{Z}_{\geq0}$ be the irregularity of $M$, which is one half of the first Betti number of $M$.
\begin{enumerate}
\item
{\rm(}Mok{\rm)} If $M$ is simply-connected, then it is holomorphically isometric to
\be\label{Mok}
(\mathbb{P}^{N_1},\theta_1)\times
\cdots\times(\mathbb{P}^{N_k},\theta_k)\times(M_1,g_1)\times\cdots
\times(M_p,g_p),\ee
where $\theta_i$ are K\"{a}hler metrics on $\mathbb{P}^{N_i}$ carrying nonnegative holomorphic bisectional curvature, and $(M_i,g_i)$ are irreducible compact Hermitian symmetric spaces of rank $\geq2$ equipped with the canonical metrics.

\item
{\rm(}Howard-Smyth-Wu{\rm)} If $\pi_1(M)$ is nontrivial, then it is infinite. Thus $q>0$ and the Albanese map $M\longrightarrow T^q_{\mathbb{C}}$ is a locally isometrically trivial holomorphic fiber bundle, where $T^q_{\mathbb{C}}$ is equipped with flat metric and the fiber is holomorphically isometric to (\ref{Mok}).
\end{enumerate}
\end{theorem}
We can now proceed to prove Theorem \ref{main application}.
\begin{proof}
First note that the Euler number $c_n(\cdot)$ of the manifolds of the form (\ref{Mok}) is strictly positive. Indeed, Since odd-dimensional homologies of irreducible compact Hermitian symmetric spaces are zero and so all $c_n(M_i)>0$. Therefore
\be\label{Eulerpositive}\begin{split}
&c_n(\mathbb{P}^{N_1}\times
\cdots\times\mathbb{P}^{N_k}\times M_1\times\cdots
\times M_p)\\
=&c_n(\mathbb{P}^{N_1})\cdots
c_n(\mathbb{P}^{N_k})c_n(M_1)\cdots c_n(M_p)>0.
\end{split}\ee

If $M$ is simply-connected, $M$ is of the form $(\ref{Mok})$, whose Euler number is strictly positive due to (\ref{Eulerpositive}). Then Theorem \ref{Euler number inequality} implies that all the Chern numbers of $M$ are positive.

If $\pi_1( M)$ is nontrivial, then by Howard-Smyth-Wu's splitting result we have nontrivial Albanese variety $T^q_{\mathbb{C}}$. Since the Ricci curvature is quasi-positive along the fiber and vanishes along $T^q_{\mathbb{C}}$, the maximal rank of Ricci curvature of $\omega$ is less than $n$. This implies that the Chern number $c_1^n[M]=0$ and consequently all the Chern numbers vanish. This completes the desired proof.
\end{proof}

\begin{remark}
We can also apply a result in \cite{DPS} to give a slightly different proof. If $c_1^n[M]>0$, then $M$ is Fano due to \cite[Prop. 3.10]{DPS}. Since a Fano manifold is simply-connected (\cite[p. 225]{Zheng00}), this reduces to the first case above. Otherwise $c_1^n[M]=0$ and this reduces to the second case.
\end{remark}

\section{Proof of Theorem \ref{level2equality}}\label{third proof}
Assume that the equality case of (\ref{class inequality-1}) in Theorem \ref{level2equality} holds. This, as we have seen in the proof of Theorem \ref{Chern class inequality theorem}, is equivalent to the second Segre class $s_2(\gamma^{\ast}S^{\ast})=0$. We want to deduce from it that either $(M,L)=\big({\mathbb P}^n, {\mathcal O}_{\mathbb{P}^n}(1)\big)$, or $N-n=1$, i.e., the Kodaira map (\ref{linear system}) embeds $M^n$ as a hypersurface in ${\mathbb P}^{n+1}$. To this end, we first introduce a quantity $d$, which is the maximum of the dimensions of \emph{the osculating spaces of order $2$} on $M$, and apply some arguments of algebro-geometric nature to show that the desired conclusion holds true if $d\leq n+1$. Then we shall show that the inequality $d\leq n+1$ follows from $s_2(\gamma^{\ast}S^{\ast})=0$.

\subsection{The osculating space of order $2$}
Let $z=(z^1, \ldots , z^n)$ be a local holomorphic coordinate system centered at $p\in U\subset M$, and $\phi: U \rightarrow {\mathbb C}^{N+1}\setminus \{ 0\}$ a local lifting of the embedding $i$ in (\ref{linear system}) around $p$. \emph{The osculating space of order $k$} at $p$ is defined to be the projective subspace $\mathbb{T}^k_p(M)$ in ${\mathbb P}^N$ passing through $p$, spanned by $[\frac{\partial \phi}{\partial z^I}(0)]$ for all multi-indices $I=(i_1, \ldots , i_n)$ with length $|I|=i_1+\cdots +i_n\leq k$. It turns out that $\mathbb{T}^k_p(M)$ is independent of the local coordinate and lifting chosen. By definition $\mathbb{T}_p(M)\subset\mathbb{T}^k_p(M)$ and $\mathbb{T}^1_p(M)=\mathbb{T}_p(M)\cong {\mathbb P}^n$ is precisely the projective tangent space at $p$.

Let $\text{Tan}(M)$ and $\text{Sec}(M)$ be the {\em tangent variety} and {\em secant variety} of $M$, which are defined to be %the (Zariski) closures of $\cup_{p\in M}\mathbb{T}_p(M)$ and the set of all lines in ${\mathbb P}^N$ joining two points in $M$, i.e.,
$$\text{Tan}(M):=\overline{\bigcup_{p\in M}\mathbb{T}_p(M)},\qquad
\text{Sec}(M):=\overline{\{\text{lines $\overline{uv}$}~|~u,v\in M,~u\neq v\}},$$
whose expected (maximal) dimensions are $2n$ and $2n+1$ respectively.

Below we focus on $\mathbb{T}^2_p(M)$, the osculating spaces of order $2$. Unlike the case of $k=1$, the dimensions of $\mathbb{T}^k_p(M)$ ($k\geq 2$) may vary and thus let $d:=\max_{p\in M}\text{dim}\mathbb{T}^2_p(M)$, the maximum of the dimensions of $\mathbb{T}^2_p(M)$ on $M$. The following lemma shows that the condition of ``$d\leq n+1$" shall yield the desired conclusion.

\begin{lemma}
If $d\leq n+1$, then either $(M,L)=\big({\mathbb P}^n, {\mathcal O}_{\mathbb{P}^n}(1)\big)$, or $N-n=1$.
\end{lemma}
\begin{proof}
It is easy to see that (cf. \cite[Lemma 1]{BF}), for a generic point $q\in\text{Tan}(M)$, where $q\in\mathbb{T}_p(M)$ for $p$ generic in $M$, $\mathbb{T}_{q}\big(\text{Tan}(M)\big)\subset\mathbb{T}^2_p(M)$. So $\text{dim}\text{Tan}(M)\leq n+1$ by the assumption. On the other hand, $\text{dim}\text{Tan}(M)\geq n$. We distinguish two different cases:
\begin{enumerate}
\item
$\text{dim}\text{Tan}(M)=n$, then all the $\mathbb{T}_{p}(M)$ coincide and thus $M=\mathbb{P}^n$.
\item
$\text{dim}\text{Tan}(M)=n+1$ and $\text{Tan}(M)=\mathbb{P}^{n+1}$, then $M$ is contained in this $\mathbb{P}^{n+1}$ and thus $n+1=N$, due to the nondegeneracy of $M$ in $\mathbb{P}^N$ \big(see the remarks after (\ref{linear system})\big).
\end{enumerate}

Now it suffices to rule out the possibility of
$\dim\text{Tan}(M)=n+1$ and $\text{Tan}(M)\neq\mathbb{P}^{n+1}$. Indeed, since $n\geq 2$,
in this case $\dim\text{Tan}(M)<2n$, the expected dimension. A beautiful result of Fulton and Hansen (\cite[Coro. 4]{FH79}) says that when $M$ is smooth and $\dim\text{Tan}(M)$ is strictly less than the expected dimension, one always has $\text{Tan}(M)=\text{Sec}(M)$. So $\dim\text{Sec}(M)=n+1$. On the other hand, by Zak's result on linear normality (\cite[Ch.II, Coro. 2.11]{Zak}), one has
$$ \dim\text{Sec}(M) \geq \frac{3}{2}n+1 = n+\frac{n}{2}+1 \geq n+2, $$
which of course contradicts $\dim\text{Sec}(M) = n+1$.
\end{proof}

\subsection{Completion of the proof}
It remains to show that, under the condition $s_2(\gamma^{\ast }S^{\ast})=0$, the inequality $d\leq n+1$ indeed holds. Note that $\gamma^{\ast }S^{\ast}$ is a quotient of the trivial bundle $\underline{\mathbb{C}}^{N+1}$, and recall from \cite{Li20} that  the induced metric on $\gamma^{\ast }S^{\ast}$ from the trivial one on $\underline{\mathbb{C}}^{N+1}$ is Bott-Chern nonnegative and hence its \emph{second Segre form} is nonnegative as a $(2,2)$-form. The condition $s_2(\gamma^{\ast}S^{\ast})=0$ then guarantees that this form is identically zero, which implies that $d\leq n+1$. For technical reason we do it on the dual bundle $\gamma^{\ast}S$, which is a subbundle of $\underline{\mathbb{C}}^{N+1}$. We shall carry out the details in the sequel.

Fix any $p\in M$. We can choose a basis $\sigma= \{ s_0, s_1, \ldots , s_N\} $ of $H^0(M,L)$ such that
$$s_0(p)\neq 0, \ \ s_1(p)=\cdots =s_N(p)=0, \ \ z^j:=\frac{s_j}{s_0},\ \ dz^1\wedge \cdots \wedge dz^n\big|_p\neq 0,$$ and
$$\frac{\partial z^{\alpha}}{\partial z^i} (p) =0 , \ \ \ \ \ \forall \ \ 1\leq i\leq n , \ \ \  \forall \ \ n+1\leq \alpha \leq N.$$
So $z=(z^1, \ldots , z^n)$ forms a local holomorphic coordinate system centered around $p$. As before the basis $\sigma$ gives us a holomorphic embedding $M \hookrightarrow {\mathbb P}^N$ and hence the Gauss map $\gamma : M \rightarrow G_{n+1}(\mathbb{C}^{N+1})$. Near $p$, the submanifold $M\subset{\mathbb P}^N$ is defined via equations
$$ z^{\alpha} = f^{\alpha}(z^1,\ldots,z^n), \ \ n+1\leq \alpha \leq N. $$
At the origin $p$, we have
$$f^{\alpha}(0)=0,\ \ \
f^{\alpha}_i (0):= \frac{\partial f^{\alpha} }{ \partial z^i}(0)=0, \ \ \
\forall \ \
n+1\leq \alpha \leq N,\ \
\forall \ \
1\leq i\leq n.$$

Let $\{ e_0, \ldots , e_N\}$ be the standard frame of the $(N+1)$-dimensional trivial bundle $\underline{\mathbb{C}}^{N+1}$ on $M$. A local frame of its subbundle $\gamma^{\ast}(S)$ near $p$ is given by $\{ X_0, X_1, \ldots , X_n\}$, where
\begin{eqnarray}
\left\{\begin{array}{ll}
X_i =e_i + \sum_{\alpha=n+1}^{N} f^{\alpha}_i e_{\alpha},~1\leq i\leq n,\\
~\\
X_0=e_0  + \sum_{\alpha=n+1}^{N} h^{\alpha }e_{\alpha},\\
~\\
h^{\alpha } := f^{\alpha} -  \sum_{j=1}^n z^jf^{\alpha}_j,~n+1\leq\alpha\leq N.
\end{array} \right.\nonumber
\end{eqnarray}

Now fix a flat metric $\langle,\rangle$ on $\underline{\mathbb{C}}^{N+1}$ so that $\{ e_0, \ldots , e_N\}$ is unitary. Denote its restricted metric on $\gamma^{\ast}S$ by $g$. Then the matrix of $g$ under the frame $\{ X_0, X_1, \ldots , X_n\}$ is
$$ g=(\langle X_i, \overline{X}_j\rangle )_{0\leq i,j\leq n} = I_{n+1} + F F^{\ast},$$
where
$$F=\begin{pmatrix}
h^{(n+1)} & h^{(n+2)} &\cdots &h^{N}\\
f^{(n+1)}_{1} & f^{(n+2)}_{1} &\cdots&f^{N}_{1}\\
\vdots&\vdots&\ddots&\vdots\\
f^{(n+1)}_{n}&f^{(n+2)}_{n}&\cdots&f^{N}_{n}
\end{pmatrix}.$$

Using the facts that $g(0)=I_{n+1}$, $dg(0)=(0)$, $F(0)=(0)$ and the entries of $F$ are holomorphic with respect to $z$, the curvature matrix $\Theta=(\Theta_{ij})_{0\leq i,j\leq n}$ of $g$ at the origin $p$ is given by
\be\begin{split}
\big(\Theta_{ij}\big)(p)&=\bar{\partial}\big[(\partial g)\cdot g^{-1}\big](0)\\
&=- \partial F \wedge (\partial F)^{\ast}(0)\\
&=-\big(\xi^{\alpha}_i\big)\wedge\big(\xi^{\beta}_{j}\big)^{\ast}(0)\\
&=\Big(-\sum_{\alpha=n+1}^{N}\xi^{\alpha}_i\wedge
\overline{\xi^{\alpha}_j}\,\Big)(0),
\end{split}\nonumber\ee
where $\xi^{\alpha}_0:=\partial h^{\alpha}$ and $\xi^{\alpha}_i:=\partial f^{\alpha}_i$ when $1\leq i\leq n$.

We now compute $s_2(\gamma^{\ast}S,g)$, the {\em second Segre form} of $\gamma^{\ast}S$ with respect to $g$, at $p$:
\be\label{second segre form}\begin{split}
&s_2(\gamma^{\ast}S,g)\\
=&\big[c_1(\gamma^{\ast}S,g)\big]^2-c_2(\gamma^{\ast}S,g)\\
=&\big[\text{tr}(\frac{\sqrt{-1}}{2\pi}\Theta)\big]^2-
\frac{\big[\text{tr}(\frac{\sqrt{-1}}{2\pi}\Theta)\big]^2-
\text{tr}\big[(\frac{\sqrt{-1}}{2\pi}\Theta)^2\big]}{2}\\
=&\frac{-1}{8\pi^2}\big[(\sum_{i=0}^n\Theta_{ii})^2+\sum_{0\leq i,j\leq n}\Theta_{ij}\Theta_{ji}\big]\\
=&\frac{-1}{4\pi^2}\big[\sum_{i=0}^n (\Theta_{ii})^2 + \sum_{0\leq i<j\leq n} (\Theta_{ii}\Theta_{jj} + \Theta_{ij}\Theta_{ji})\big]\\
=&\frac{1}{4\pi^2}\big[\sum_{i, \alpha , \beta} \xi^{\alpha}_i\wedge\xi^{\beta}_i\wedge \overline{\xi^{\alpha}_i\wedge\xi^{\beta}_i} + \frac{1}{2}\sum_{i<j, \alpha, \beta} (\xi^{\alpha}_i\wedge \xi^{\beta}_j - \xi^{\beta}_i\wedge\xi^{\alpha}_j) \wedge (\overline{\xi^{\alpha}_i\wedge \xi^{\beta}_j - \xi^{\beta}_i\wedge\xi^{\alpha}_j})\big]\\
=:&\frac{1}{4\pi^2}(A+B),
\end{split}\nonumber\ee
which is a nonnegative $(2,2)$-form as so are $A$ and $B$. Now if $s_2(\gamma^{\ast}S^{\ast})=0$, then $s_2(\gamma^{\ast}S)=0$ and so $A=B=0$.  At the origin we have
$$ \xi^{\alpha}_0(0)=\partial h^{\alpha}(0)=0, \ \ \xi^{\alpha}_i(0) = \partial f^{\alpha}_i (0) =: \sum_{j=1}^n f^{\alpha}_{ij}(0)dz^j.$$
Write $ H^{\alpha} = \big(f^{\alpha}_{ij}(0)\big)$ for the Hessian matrices. Our goal is to show that the linear space $W={\mathbb C} \{ H^{n+1}, \ldots , H^N\}$ spanned by  these Hessians at $p$ is at most one dimensional. For any $v\in V:={\mathbb C}^n$, let us write $H^{\alpha}_v = \xi^{\alpha}_v(0) = \sum_{i=1}^n v_i \xi^{\alpha}_i(0)$, which can be viewed as a vector in $V$ (under the coframe $dz^j$). Note that $A=B=0$ means
\begin{equation}
H^{\alpha}_v \wedge H^{\beta}_v=0, \ \ \ \forall \  n+1\leq \alpha, \beta \leq N, \ \forall \ v\in V. \label{parallel}
\end{equation}
Given any $u$, $v\in V$,  if we replace $v$ in (\ref{parallel}) by $u+tv$, where $t\in {\mathbb C}$, and look at the $t$-terms, we get
\begin{equation}
H^{\alpha}_u \wedge H^{\beta}_v  + H^{\alpha}_v \wedge H^{\beta}_u = 0. \label{parallel2}
\end{equation}
To see that $W$ has dimension at most one, first let us assume that there is an $\alpha$ such that the rank of $H^{\alpha}$ is at least $2$. Thus $H^{\alpha}_u \wedge H^{\alpha}_v \neq 0$ for generic $u, v \in V$. Let $U$ be the open dense subset of $V$ such that $H^{\alpha}_u \wedge H^{\alpha}_v \neq 0$ for any $u,v\in U$.  For each $v\in U$, since $H^{\alpha}_v\neq 0$, by the equation (\ref{parallel}), we know that there exists a unique constant $\lambda (v)$ such that $ H^{\beta}_v = \lambda (v) H^{\alpha}_v$. Also, for any $u$, $v$ in $U$, by (\ref{parallel2}) we get
\begin{equation}
(\lambda (u)-\lambda (v)) H^{\alpha}_u \wedge H^{\alpha}_v  =0,
\end{equation}
hence $\lambda (u) = \lambda (v)$. This means that $\lambda$ is a constant function on $U$, hence we have $H^{\beta}=\lambda H^{\alpha}$. That is, any other Hessian is a constant multiple of this $H^{\alpha}$.

Now we are left with the case when each $H^{\alpha}$, or any of there linear combinations, has rank at most one. We know that for each $i$, the $i$-th rows of these matrices are all parallel, and all these matrices are symmetric, so each of them is a constant multiple of $\,v\,^t\!v$ for some fixed column vector $v$ in $V$. So all these Hessian matrices at $p$ form a linear space of dimension at most one. By definition, this means exactly that the second osculating space $\mathbb{T}^2_p(M)$ at $p$ is at most $(n+1)$-dimensional. This completes the proof of Theorem \ref{level2equality}.

\end{document}